\documentclass[10pt]{amsart}
\usepackage{amsmath,amssymb,latexsym, amsfonts, amscd, amsthm}
\usepackage{color}

\usepackage{tikz-cd}

\usepackage{tikz}
\theoremstyle{plain}

\newtheorem{theorem}{Theorem}[section]

\newtheorem{proposition}[theorem]{Proposition}
\newtheorem{corollary}[theorem]{Corollary}

\theoremstyle{definition}
\newtheorem{definition}[theorem]{Definition}
\newtheorem{remark}[theorem]{Remark}

\def\bdf{\begin{defn}}
\def\edf{\end{defn}}

\usepackage{color}

\begin{document}

\title{A remark on the Kottwitz homomorphism}

\author{Moshe Adrian}
\address{Department of Mathematics 
Queens College, CUNY
65-30 Kissena Blvd., Queens, NY 11367-15971}
\email{moshe.adrian@qc.cuny.edu} 

\maketitle

\begin{abstract}
We prove that for any split, almost simple, connected reductive group $G$ over a $p$-adic field $F$, the Kottwitz homomorphism $\kappa : G(F) \rightarrow \Omega$ exhibits a homomorphic section $\Omega \hookrightarrow G(F)$.  We then extend this result to certain additional split connected reductive groups.
\end{abstract}

\section{Introduction}

Let $G$ be a connected reductive group over a $p$-adic field $F$.  In \cite{Kot97}, Kottwitz defined a canonical homomorphism 
\[
\kappa : G(F) \twoheadrightarrow X^*(Z(\widehat{G})^I)^{\mathrm{Fr}}.
\]
This homomorphism is surjective and, in the case that $G$ is split, simplifies to a homomorphism
\[
\kappa : G(F) \twoheadrightarrow X^*(Z(\widehat{G})) \cong X_*(T) / Q^{\vee}.
\]

In this note, we show that the map $\kappa$ has a homomorphic section in the case that $G$ is split and almost simple, as well as for certain additional split groups.  More specifically, fix a fundamental alcove in the building of $G$ corresponding to a maximal split torus $T$, and let $\Omega$ be the subgroup of the extended affine Weyl group $W$ that stabilizes $C$.  We show that there is a homomorphic section of the canonical projection $N_G(T) \twoheadrightarrow \Omega$, where $N_G(T)$ is the normalizer of a maximal torus $T$ in $G$.  If $G$ is almost-simple, then this section can be described as follows: it is known (see Proposition \ref{bijection}) that $\Omega$ may be identified with a collection of elements $\{1, \epsilon_i \rtimes w_i \} \subset W = X_*(T) \rtimes W_{\circ}$, where $\epsilon_i$ are certain fundamental coweights and $W_{\circ}$ is the finite Weyl group.  By \cite[\S9.3.3]{Spr98}, there is a canonical map $\mathcal{N}_{\circ} : W_{\circ} \rightarrow N_G(T)$ (denoted $\phi$ in $loc. \ cit.$) that is compatible with the projection $N_G(T) \rightarrow W_{\circ}$.  We may then consider the map
\[
\iota : \Omega \rightarrow N_G(T)
\]
\[
\ \ \ \ \ \ \ \ \ \ \  \epsilon_i w_i \mapsto \epsilon_i(\varpi^{-1}) \mathcal{N}_{\circ}(w_i),
\]
where $\varpi$ is a uniformizer in $F$.
The map $\iota$ is a section of the projection $N_G(T) \twoheadrightarrow \Omega$, and it turns out that $\iota$ is a homomorphism in all cases except the adjoint group of type $D_l$ where $l$ is odd, and some cases in type $A_l$ (see Theorem \ref{mainresult} and Remark \ref{nonhomomorphic}).  Nonetheless, we can still use $\iota$ to construct a homomorphic section for all almost-simple $p$-adic groups (see Theorem \ref{mainresult}).  We then show that for certain split connected groups with connected center, the Kottwitz homomorphism exhibits a homomorphic section (see Proposition \ref{connectedcenter}).

We would like to remark that if $G$ is any split connected reductive group with simply connected derived group, then $\kappa$ has a homomorphic section.  This follows from the fact that $\Omega$ is a free abelian group of finite rank, isomorphic to a free quotient of $X_*(T)$.  Then one constructs a section by taking a homomorphic section of $X_*(T) \rightarrow \Omega$ and then composing that section with the map $X_*(T) \rightarrow T$, $\lambda \mapsto \lambda(\varpi^{-1})$.  In particular, the image of this section lies in $T$, not just $N_G(T)$.  The situation where $\Omega$ is \emph{finite} is much more subtle, which is what this paper is about.

\subsection{Acknowledgements}
This paper was written in response to a question that Karol Koziol asked me; I wish to thank him for asking the question.  I thank the referees for their comments, especially one of the referees for very valuable comments and suggestions; in particular, ideas on how to expand the results beyond split almost-simple groups, and for the proof of Proposition \ref{connectedcenter}.  I thank Karol Koziol and Sean Rostami for helpful conversations.  Support for this project was provided by a grant from the Simons
Foundation \#422638 and by a PSC-CUNY award, jointly funded by the
Professional Staff Congress and The City University of New York.

\section{Preliminaries}
Let $G$ be a split connected reductive group over a $p$-adic field $F$.  Fix a pinning $(B, T, \{ X_{\alpha} \})$ for $G$.  This gives rise to a set of non-zero roots $\Phi$ of $G$ with respect to $T$,  a set of positive roots $\Pi$ in $\Phi$, and a basis $\Delta = \{\alpha_1, \alpha_2, ..., \alpha_l \}$ of the set of positive roots, so that $l$ is the rank of $G$.  We recall that for each $\alpha \in \Phi$, there exists an isomorphism $u_{\alpha}$ of $F$ onto a unique closed subgroup $U_{\alpha}$ of $G$ such that $t u_{\alpha}(x) t^{-1} = u_{\alpha}(\alpha(t) x)$, for $t \in T, x \in F$ \cite[\S8.1.1]{Spr98}.

Let $X^*(T)$, $X_*(T)$ be the character, cocharacter lattices of $T$, respectively.  Let $Q$ be the lattice generated by $\Phi$, and $P^{\vee}$ the coweight lattice.   Namely, $P^{\vee}$ is the $\mathbb{Z}$-dual of $Q$ relative to the standard pairing 
$( \cdot, \cdot) : X^*(T) \times X_*(T) \rightarrow \mathbb{Z}$.   We let $\Phi^{\vee}$ be the system of coroots, $Q^{\vee}$ the lattice generated by $\Phi^{\vee}$, and $P$ the weight lattice.  Then $P^{\vee}$ is spanned by the $l$ fundamental coweights, which are denoted $\epsilon_1, \epsilon_2, ..., \epsilon_l$.  We recall that the $\epsilon_i$ are defined by the relation $(\epsilon_i, \alpha_j) = \delta_{ij}$.  If $\alpha$ is a root, we denote its associated coroot by $\alpha^{\vee}$.

We now let $W_{\circ} = N_G(T) / T$ be the Weyl group of $G$ relative to $T$.  For each $\alpha \in \Phi$, we let $s_{\alpha} \in W_{\circ}$ be the simple reflection associated to $\alpha$.  Then the $u_{\alpha}$ may be chosen such that for all $\alpha \in R$, $n_{\alpha} = u_{\alpha}(1) u_{-\alpha}(-1) u_{\alpha}(1)$ lies in $N_G(T)$ and has image $s_{\alpha}$ in $W_{\circ}$ (see \cite[\S8.1.4]{Spr98}).  Relative to the pinning that we have chosen, there is a canonical, well-defined map $\mathcal{N}_{\circ} : W_{\circ} \rightarrow N_G(T)$ \cite[\S9.3.3]{Spr98} (the map is denoted $\phi$ in loc.cit.), defined by $\mathcal{N}_{\circ}(w) = n_{\beta_1} n_{\beta_2} \cdots n_{\beta_m}$ for a reduced expression $w = s_{\beta_1} s_{\beta_2} \cdots s_{\beta_m}$.

\subsection{The map $\mathcal{N}_{\circ}$}\label{rostami}

In this section, we recall a result about the map $\mathcal{N}_{\circ}$ from \cite{Ros16}.

\begin{definition}
For $u,v \in W_{\circ}$, we define 
\[
\mathcal{F}(u,v) = \{\alpha \in \Pi \ | \ v(\alpha) \in -\Pi, u(v(\alpha)) \in \Pi \}.
\]
\end{definition}
The following proposition describes the failure of $\mathcal{N}_{\circ}$ to be a homomorphism.
\begin{proposition}\cite[Proposition 3.1.2]{Ros16}\label{ahomomorphicity}
For $u,v \in W_{\circ}$,
\[
\mathcal{N}_{\circ}(u) \cdot \mathcal{N}_{\circ}(v) = \mathcal{N}_{\circ}(u \cdot v) \cdot \displaystyle\prod_{\alpha \in \mathcal{F}(u,v)} \alpha^{\vee}(-1).
\]
\end{proposition}

\begin{definition}
For $w \in W_{\circ}, i \in \mathbb{N}$, we define
\[
\mathcal{F}_w(i) = \{ \alpha \in \Pi \ | \ w^i(\alpha) \in -\Pi, w^{i+1}(\alpha) \in \Pi  \}.
\]
\end{definition}

\begin{corollary}\label{powersformula}
If $w \in W_{\circ}$ and $n \in \mathbb{N}$, then
$$\mathcal{N}_{\circ}(w)^n = \mathcal{N}_{\circ}(w^n) \cdot \displaystyle\prod_{m = 1}^{n-1} \displaystyle\prod_{\alpha \in \mathcal{F}_w(m)} \alpha^{\vee}(-1).$$
\end{corollary}

\begin{proof}
By Proposition \ref{ahomomorphicity}, $\mathcal{N}_{\circ}(w)^2 = \mathcal{N}_{\circ}(w^2) \cdot \displaystyle\prod_{\alpha \in \mathcal{F}_w(1)} \alpha^{\vee}(-1).$  Multiplying by $\mathcal{N}_{\circ}(w)$ on the left and using Proposition \ref{ahomomorphicity} again, we get $\mathcal{N}_{\circ}(w)^3 = \mathcal{N}_{\circ}(w^3) \cdot \displaystyle\prod_{\alpha \in \mathcal{F}_w(2)} \alpha^{\vee}(-1)  \cdot \displaystyle\prod_{\alpha \in \mathcal{F}_w(1)} \alpha^{\vee}(-1).$ Continuing in this way, the claim follows.
\end{proof}

\section{Embedding $\Omega$ into $G$}

Let $G$ be a split, almost-simple $p$-adic group.  We set $W = N_G(T) / T_{\circ}$, where $T_{\circ}$ is the maximal bounded subgroup of $T$.  The group $W$ is the extended affine Weyl group, and we note that we have a semidirect product decomposition $W = X_*(T) \rtimes W_{\circ}$.  
We also set $\Omega = W / W^{\circ}$, where $W^{\circ} = Q^{\vee} \rtimes W_{\circ}$ is the affine Weyl group.  We therefore have a canonical projection $N_G(T) \rightarrow \Omega$.  This projection is exactly the restriction of $\kappa$ to $N_G(T)$.

The group $\Omega$ can be identified with the subgroup of $W$ that stabilizes a fundamental alcove $\mathcal{C}$.  Moreover, it is known that $\Omega$ acts on the set $\{1 - \alpha_0, \alpha_1, \alpha_2, ..., \alpha_l \}$, where $\alpha_0$ is the highest root in $\Phi$.  The action of $\Omega$ on this set can be found in \cite[p. 18-19]{IM65}.  We let $\Omega_{\mathrm{ad}}$ be the analogous group for the adjoint group $G_{\mathrm{ad}}$.

It is known that there exists in $W_{\circ}$ an element $w_{\Delta}$ such that $w_{\Delta}(\Delta) = -\Delta$.  The element $w_{\Delta}$ is unique and satisfies $w_{\Delta}^2 = 1$.  Moreover, if we denote the subset $\Delta - \{\alpha_i \}$ by $\Delta_i$, then the subgroup $W_i$ of $W_{\circ}$ generated by $s_{\alpha_1}, ..., \hat{s}_{\alpha_i}, ..., s_{\alpha_l}$ ($\hat{s}_{\alpha_i}$ means that $s_{\alpha_i}$ is omitted) contains an element $w_{\Delta_i}$ such that $w_{\Delta_i}(\Delta_i) = -\Delta_i$ and $w_{\Delta_i}^2 =1$.  

We recall the following result from \cite{IM65}.

\begin{proposition}\cite[Proposition 1.18]{IM65}\label{bijection}
The mapping from the set $\{0\} \cup \{\epsilon_i : (\alpha_0, \epsilon_i) = 1 \}$ onto $\Omega_{\mathrm{ad}}$ defined by $0 \mapsto 1, \epsilon_i \mapsto \epsilon_i w_{\Delta_i} w_{\Delta}$ is bijective.
\end{proposition}

The notation $\rho_i$ (and sometimes $\rho$) is used in \cite{IM65} to denote the element $\epsilon_i w_{\Pi_i} w_{\Delta}$.  We will adopt the same notation.  We will also let $S_{\mathrm{ad}}$ denote the set $\{0\} \cup \{\epsilon_i : (\alpha_0, \epsilon_i) = 1 \}$.  We note that every lattice between $Q^{\vee}$ and $P^{\vee}$ arises as $\langle Q^{\vee}, S \rangle$, for some subset $S \subset S_{\mathrm{ad}}$.  If $S$ is such a subset, we will talk of the almost-simple $p$-adic group $G$ that is determined by the lattice $\langle Q^{\vee}, S \rangle$.  We note in particular that if $\Omega_G$ denotes the omega group for $G$, then one can see that $\Omega_G = W / W^{\circ} \cong \langle 1, \rho_i : \epsilon_i \in S \rangle$.

We now assume that $G$ is not simply connected.  For otherwise, $\Omega = 1$, so the claim that $\kappa : G \rightarrow \Omega$ has a homomorphic section is vacuous.

Let $\varpi$ be a uniformizer of $F$.  There is a natural map $X_*(T) \rightarrow N_G(T)$ given by $\lambda \mapsto \lambda(\varpi^{-1})$ (see \cite[p. 31]{Tit79}).  We also have the map $\mathcal{N}_{\circ} : W_{\circ} \rightarrow N_G(T)$.  Coupling these maps together, we obtain a natural map
\[
W \rightarrow N_G(T)
\]
\[
\ \ \ \ \ (\lambda, w) \mapsto \lambda(\varpi^{-1}) \mathcal{N}_{\circ}(w)
\]
for $\lambda \in X_*(T), w \in W_{\circ}$. Most of the time, we will write $\lambda w$ instead of $(\lambda, w)$.  Proposition \ref{bijection} gives us a set-theoretic embedding $\Omega \hookrightarrow W$.  We can then consider the composite map $\Omega \hookrightarrow W \rightarrow N_G(T)$, which gives us a section of the canonical projection $N_G(T) \rightarrow \Omega$:
$$\iota : \Omega \rightarrow N_G(T)$$ $$\omega = \epsilon_i w_{\Delta_i} w_{\Delta} \mapsto \epsilon_i(\varpi^{-1}) \mathcal{N}_{\circ}(w_{\Delta_i} w_{\Delta})$$ 

That $\epsilon_i(\varpi^{-1})$ is well-defined follows from the fact that $\epsilon_i \in X_*(T)$ by our definition of $G$ earlier.    That $\iota$ is a section follows from Proposition \ref{bijection}.  In particular, $\iota$ is injective. We will sometimes identify $\epsilon_i$ with $\epsilon_i(\varpi^{-1})$ for ease of notation.

We will show that $\iota$ is a homomorphic embedding for all types except $A_l$, and the specific case when $G$ is adjoint of type $D_l$ where $l$ is odd.  Nonetheless, we will still produce a homomorphic embedding $\Omega \hookrightarrow N_G(T)$ which is a section of $N_G(T) \rightarrow \Omega$, in these two outlier cases.

Suppose $\omega = \epsilon_i w_{\Delta_i} w_{\Delta}$ is a generator of $\Omega$, whose order is $r$.    Propositions \ref{firstprop} and \ref{secondprop} will be dedicated to showing that $\iota(\omega)$ also has order $r$.  Let $w_i = w_{\Delta_i} w_{\Delta}$ for convenience of notation.  We compute 
\[
\iota(\omega)^r = (\epsilon_i(\varpi^{-1}) \mathcal{N}_{\circ}(w_i))^r = \epsilon_i(\varpi^{-1}) \cdot \left( \mathcal{N}_{\circ}(w_i) \epsilon_i(\varpi^{-1}) \mathcal{N}_{\circ}(w_i)^{-1} \right) \cdot \left(\mathcal{N}_{\circ}(w_i)^2 \epsilon_i(\varpi^{-1}) \mathcal{N}_{\circ}(w_i)^{-2}\right) \]
\[
 \cdots\left(\mathcal{N}_{\circ}(w_i)^{r-1} \epsilon_i(\varpi^{-1}) \mathcal{N}_{\circ}(w_i)^{1-r} \right) \mathcal{N}_{\circ}(w_i)^r = (\epsilon_i + w_i(\epsilon_i )+ w_i^2(\epsilon_i) + \cdots + w_i^{r-1}(\epsilon_i))(\varpi^{-1}) \mathcal{N}_{\circ}(w_i)^r.
\]
We will now show that $\epsilon_i + w_i(\epsilon_i )+ w_i^2(\epsilon_i) + \cdots + w_i^{r-1}(\epsilon_i) = 0$ and $ \mathcal{N}_{\circ}(w_i)^r = 1$.
\begin{proposition}\label{firstprop}
$\epsilon_i + w_i(\epsilon_i )+ w_i^2(\epsilon_i) + \cdots + w_i^{r-1}(\epsilon_i) = 0$.
\end{proposition}

\begin{proof}
We compute $w_i^j(\epsilon_i)$ for $j = 1, 2, ..., r-1$.  The tables on pages 18-19 of \cite{IM65} give the values of $i$ for each type, and the explicit action of $w_i$ on the set $\{ -\alpha_0, \alpha_1, \alpha_2, ..., \alpha_l \}$.    We also note that the order of $\omega$ equals the order of $w_i$ (see \cite[p. 18]{IM65}).  We begin with type $B_l$ and end with type $A_l$ (since, computationally, $A_l$ is the most intricate).

\begin{itemize}
\item  In type $B_l$, we have that $r = 2$ and $i = 1$, so we wish to show that $\epsilon_1 + w_1(\epsilon_1) = 0$.  Note that $w_1(\alpha_1) = -\alpha_0$ and $w_1$ fixes the other simple roots.  To compute $w_1(\epsilon_1)$, we pair $w_1(\epsilon_1)$ with all of the simple roots.  By Weyl-invariance of the inner product $(\cdot, \cdot)$ and the fact that $w_1^2 = 1$, we have
\begin{equation*}
(w_1(\epsilon_1), \alpha_j) = \left\{
\begin{array}{rll}
0 & \text{if} & j \neq 1\\
(\epsilon_1, -\alpha_0) & \text{if} & j = 1
\end{array} \right.
\end{equation*}
But since $\alpha_0 = \alpha_1 + 2(\alpha_2 + ... + \alpha_l)$, we have that $(\epsilon_1, -\alpha_0) = -1$.  Therefore, $w_1(\epsilon_1) = -\epsilon_1$, so that $\epsilon_1 + w_1(\epsilon_1) = 0$.  
 
\item  In type $C_l$, we have that $i = l$, and the same argument as in type $B_l$ holds.  Indeed, $w_l(\alpha_l) = -\alpha_0$, $w_l$ permutes the simple roots other than $\alpha_l$, and $\alpha_0 = 2(\alpha_1 + ... + \alpha_{l-1}) + \alpha_l$.  Thus, $(w_l(\epsilon_l), \alpha_l) = -1$, so $w_l(\epsilon_l) = -\epsilon_l$, and the result follows.
 
\item We consider type $D_l$.  We first consider the case that $l$ is odd and $G$ is adjoint.  In this case $\Omega \cong \mathbb{Z} / 4 \mathbb{Z}$ and it is enough to consider $i = l$.  The claim is that $\epsilon_l + w_l(\epsilon_l) + w_l^2(\epsilon_l) + w_l^3(\epsilon_l) = 0$.  One can see from the table on \cite[p. 19]{IM65} that $w_l$ permutes $\alpha_2, \alpha_3, ..., \alpha_{l-2}$, and  also acts by $-\alpha_0 \mapsto \alpha_l \mapsto \alpha_1 \mapsto \alpha_{l-1} \mapsto -\alpha_0$. We therefore conclude that $(w_l(\epsilon_l), \alpha_j) = (\epsilon_l, w_l^3 (\alpha_j)) = 0$ if $j = 2, 3, ..., l - 2$.  Moreover, since $w_l^3(\alpha_1) = \alpha_l, w_l^3(\alpha_{l-1}) = \alpha_1$,  $w_l^3(\alpha_l) = -\alpha_0$, we conclude that $(w_l(\epsilon_l), \alpha_1) = 1, (w_l(\epsilon_l), \alpha_{l-1}) = 0,$ and $(w_l(\epsilon_l), \alpha_l) = -1$.  Therefore, $w_l(\epsilon_l) = \epsilon_1 - \epsilon_l$.  One can compute similarly that $w_l^2(\epsilon_l) = \epsilon_{l-1} - \epsilon_1$ and $w_l^3(\epsilon_l) = -\epsilon_{l-1}$.  Therefore, $\epsilon_l + w_l(\epsilon_l) + w_l^2(\epsilon_l) + w_l^3(\epsilon_l) = 0$.
 
We now consider the case where $l$ is odd and $G$ is neither adjoint nor simply connected.  We have that $\rho_l^2 = \rho_1$ generates $\Omega$.  Thus, we need to show that 
$\epsilon_1 + w_1(\epsilon_1) = 0$. First, we note that $w_1$ fixes $\alpha_j$, for $j = 2, 3, ..., l-2$, it exchanges $-\alpha_0$ and $\alpha_1$, and it exchanges $\alpha_{l-1}$ and $\alpha_l$.  Since $w_1$ has order $2$, we compute that
\[
(w_1(\epsilon_1), \alpha_j) = (\epsilon_1, w_1(\alpha_j)) = 0
\]
if $j = 2, ..., l$.  We also have $(w_1(\epsilon_1), \alpha_1) = (\epsilon_1, -\alpha_0) = -1$.  Thus, $w_1(\epsilon_1) = -\epsilon_1$, so the result follows.

We now consider the case that $l$ is even and $G$ is adjoint.  In this case, $\Omega \cong \mathbb{Z} / 2 \mathbb{Z} \times \mathbb{Z} / 2 \mathbb{Z}$.  In the notation of \cite[p. 19]{IM65}, the generators of $\Omega$ are $\rho_1, \rho_{l-1}, \rho_l$.  It is straightforward to compute that $w_1(\epsilon_1) = -\epsilon_1, w_l(\epsilon_{l}) = -\epsilon_l$, and that $w_{l-1}(\epsilon_{l-1}) = -\epsilon_{l-1}$, proving the claim for $G$.  

If $l$ is even and $G$ is neither adjoint nor simply connected, the result follows readily from the adjoint case.

\item We now consider type $E_6$.  Then $r = 3, i = 1$, and $w_1$ acts by $\alpha_1 \mapsto \alpha_6 \mapsto -\alpha_0$.  Since $w_1$ has order $3$, we compute that

\begin{equation*}
(w_1(\epsilon_1), \alpha_j) = (\epsilon_1, w_1^2(\alpha_j)) =  \left\{
\begin{array}{rll}
0 & \text{if} & j \neq 1,6\\
-1 & \text{if} & j = 1\\
1 & \text{if} & j = 6
\end{array} \right.
\end{equation*}
which implies that $w_1(\epsilon_1) = -\epsilon_1 + \epsilon_6$.  Similarly one may compute that $w_1^2(\epsilon_1)  = -\epsilon_6$.  Therefore, $\epsilon_1 + w_1(\epsilon_1) + w_1^2(\epsilon_1) = 0$.
 
\item Type $E_7$ is analogous to types $B_l$ and types $C_l$.  Just note that in this case we have $i = 1$ and $w_1(\alpha_1) = -\alpha_0$, and from  \cite[p. 19]{IM65}  we see that the coefficient of $\alpha_1$ in $\alpha_0$ is $1$.

\item We finally consider type $A_l$.  We may identify roots and co-roots, fundamental weights and fundamental co-weights.  Recall that we may take $\Delta = \{\alpha_1, \alpha_2, ..., \alpha_l \}$ to be $\alpha_i = \alpha_i^{\vee} = e_i - e_{i+1}$ for $1 \leq i \leq l$.  The corresponding fundamental coweights are
\begin{equation*}
\epsilon_i = \epsilon_i^{\vee} = \left\{
\begin{array}{lll}
\frac{1}{l+1} [(l+1-i) (e_1 + e_2 + ... + e_i) - i(e_{i+1} + e_{i+2} + ... + e_{l+1})] & \text{if} & i \leq l
\end{array} \right.
\end{equation*}
We recall that the isogenies of type $A_l$ are in one to one correspondence with the subgroups of $\Omega_{\mathrm{ad}} = \mathbb{Z} / (l + 1) \mathbb{Z}$.  The element $\rho_1$ generates $\Omega_{\mathrm{ad}}$.  Let $a,b \in \mathbb{N}$ such that $l + 1 = ab$.  Let $\omega = \rho_1^a$, so that $\omega^b = 1$.  Let $G$ be the group of type $A_l$ that is given by the subgroup $\langle \omega \rangle$ of $\Omega_{\mathrm{ad}}$.  In particular, its associated cocharacter lattice, which we denote by $X_*(A_l^a)$, is given by $\langle Q^{\vee}, \epsilon_a \rangle$.  For ease of notation, let $n = l+1$.  Then $\omega = \epsilon_a w_a$, where $w_a$ is the a-th power of the $n$-cycle $(1 \ 2 \ \cdots \ n)$.  We  need to show that  $\epsilon_a + w_a(\epsilon_a) + ... + w_a^{b-1}(\epsilon_a) = 0$.  A computation shows that this sum is
\[
\frac{1}{n}[ (n-a)(e_1 + e_2 + ... + e_a) - a(e_{a+1} + ... + e_n)] 
\]
\[ 
+ \frac{1}{n}[ (n-a)(e_{a+1} + e_{a+2} + ... + e_{2a}) - a(e_{2a+1} + ... + e_n + e_1 + e_2 + ... + e_a)]
\]

\[
+ ... + \frac{1}{n} [ (n-a) (e_{n-a+1} + e_{n-a+2} + ... + e_n) - a(e_1 + e_2 + ... + e_{n-a})],
\]
which equals zero.

\end{itemize}
 
\end{proof}
 
\begin{proposition}\label{secondprop}
$ \mathcal{N}_{\circ}(w_i)^r = 1$.
\end{proposition} 
 
\begin{proof} 
We again assume that $G$ is not simply connected.  We proceed on a type by type basis, beginning with $B_l$ and ending again with type $A_l$.  To compute $\mathcal{N}_{\circ}(w_i)^r$, we use Corollary \ref{powersformula}.  We remind the reader again that the order of $\epsilon_i w_i$ equals the order of $w_i$ (see \cite[p. 18]{IM65}).
\begin{itemize}
\item Suppose that $G$ is of type $B_l$ and adjoint.  We recall that the roots may be identified with the functionals $\pm  e_i (1 \leq i \leq l)$ and $\pm e_i \pm e_j (1 \leq i < j \leq l)$.  The corresponding coroots may be identified (in the obvious way) with the functionals $\pm 2e_i, \pm e_i \pm e_j$.  
The fundamental coweights corresponding to the standard choice of simple roots are given by $\epsilon_i = e_1 + ... + e_i$, for $1 \leq i \leq l$.

We note that the cocharacter lattice of type $B_l$ adjoint is $\langle Q^{\vee}, \epsilon_1 \rangle$.  The action of $w_1$ exchanges $-\alpha_0$ and $\alpha_1$.  Therefore, $\mathcal{F}_{w_1}(1)$ is the set of positive roots that contain $\alpha_1$.  In other words, $\mathcal{F}_{w_1}(1) = \{ e_1 + e_j : j = 2,3,...,l \} \cup \{ e_1 - e_j : j = 2, 3, ..., l \} \cup \{e_1 \}$.  One may therefore compute that 
\[
\displaystyle\sum_{\alpha \in \mathcal{F}_{w_1}(1)} \alpha^{\vee} = \displaystyle\sum_{j > 1} (e_1 + e_j)^{\vee} + \displaystyle\sum_{j > 1} (e_1 - e_j)^{\vee} + e_1^{\vee} = \displaystyle\sum_{j > 1} (e_1 + e_j) + \displaystyle\sum_{j > 1} (e_1 - e_j) + 2e_1,
\]
where we have identified $e_1^{\vee}$ with $2e_1$ in the usual way.  Writing $e_1 + e_j$ and $e_1 - e_j$ as sums of simple coroots, one may compute that 

\[
 \displaystyle\sum_{j > 1} (e_1 + e_j) + \displaystyle\sum_{j > 1} (e_1 - e_j) + 2e_1 = 2l \alpha_1^{\vee} + 2l \alpha_2^{\vee} + ... + 2l \alpha_{l-1}^{\vee} + l \alpha_l^{\vee}.
 \]

Noting that $\epsilon_1 = \alpha_1^{\vee} + \alpha_2^{\vee} + ... + \alpha_{l-1}^{\vee} + \frac{1}{2} \alpha_l^{\vee}$, we have that $\displaystyle\sum_{\alpha \in \mathcal{F}_{w_1}(1)} \alpha^{\vee} = 2l \epsilon_1$.  Therefore, $\mathcal{N}_{\circ}(w_1)^2 = (\epsilon_1)(-1)^{2l} =  1$.

\item We now turn to type $C_l$ adjoint.  We recall that the roots may be identified with the functionals $\pm 2 e_i (1 \leq i \leq l)$ and $\pm e_i \pm e_j (1 \leq i < j \leq l)$.  The corresponding coroots may be identified (in the obvious way) with the functionals $\pm e_i, \pm e_i \pm e_j$.  The fundamental coweights corresponding to the standard choice of simple roots are $\epsilon_i = e_1 + ... + e_i$, for $1 \leq i < l$, and $\epsilon_l = \frac{1}{2}(e_1 + e_2 + ... + e_l)$.

We note that the cocharacter lattice of type $C_l$ adjoint is $\langle Q^{\vee}, \epsilon_l \rangle$.  
The action of $w_l$ exchanges $-\alpha_0$ and $\alpha_l$.  Therefore, $\mathcal{F}_{w_l}(1)$ is the set of all positive roots that contain $\alpha_l$, so that $\mathcal{F}_{w_l}(1) = \{ e_i + e_j : 1 \leq i < j \leq l \} \cup \{2 e_i : 1 \leq i \leq l \}$.  Therefore, one may compute that 
\[
\displaystyle\sum_{\alpha \in \mathcal{F}_{w_l}(1)} \alpha^{\vee} = \displaystyle\sum_{i < j} (e_i + e_j)^{\vee} + \displaystyle\sum_i (2e_i)^{\vee} = \displaystyle\sum_{i < j} (e_i + e_j) + \displaystyle\sum_{i} e_i.
\]
Writing $e_i + e_j$ and $e_i$ as sums of simple coroots, we may compute that 
\[
\displaystyle\sum_{i < j} (e_i + e_j) + \displaystyle\sum_{i} e_i = l \alpha_1^{\vee} + 2l \alpha_2^{\vee} + 3l \alpha_3^{\vee} + ... + l^2 \alpha_l^{\vee}.
\]
Recalling that $\epsilon_l = \frac{1}{2} (\alpha_1^{\vee} + 2 \alpha_2^{\vee} + 3 \alpha_3^{\vee} + ... + l \alpha_l^{\vee})$, we have that $\displaystyle\sum_{\alpha \in \mathcal{F}_{w_l}(1)} \alpha^{\vee} = 2l \epsilon_l$.  Therefore, $\mathcal{N}_{\circ}(w)^2 = (\epsilon_l)(-1)^{2l} = 1$.

\item We now consider type $D_l$.  We recall that the root system of type $D_l$ is realized as the set of all $\pm e_i \pm e_j$, with $i < j$.  Since all roots $\alpha$ satisfy $||\alpha||^2 = 2$, we may identify roots and co-roots, fundamental weights and fundamental co-weights.  Recall that we may take $\Delta = \{\alpha_1, \alpha_2, ..., \alpha_l \}$ to be
\begin{equation*}
\alpha_i = \alpha_i^{\vee} = \left\{
\begin{array}{rll}
e_i - e_{i+1} & \text{if} & i \leq l - 1\\
e_{l-1} + e_l &  \text{if} & i = l
\end{array} \right.
\end{equation*}

The corresponding fundamental weights are

\begin{equation*}
\epsilon_i = \epsilon_i^{\vee} = \left\{
\begin{array}{lll}
e_1 + ... + e_i & \text{if} & i < l - 1\\
\frac{1}{2} (e_1 + ... + e_{l-1} - e_l) &  \text{if} & i = l-1 \\
\frac{1}{2} (e_1 + ... + e_{l-1} + e_l) &  \text{if} & i = l
\end{array} \right.
\end{equation*}

First we consider the case where $l$ is even.   By \cite[p. 19]{IM65}, $\Omega_{\mathrm{ad}}$ is generated by the elements $\rho_1, \rho_{l-1}, \rho_l$, and the actions of their corresponding Weyl elements $w_1, w_{l-1}, w_l$ on the set $\{-\alpha_0, \alpha_1, \alpha_2, ..., \alpha_l \}$ are given by
$$w_1 (-\alpha_0) = \alpha_1, \ \ w_1(\alpha_1) = -\alpha_0, \ \ w_1 (\alpha_i) = \alpha_i \ \ \ \ (2 \leq i \leq l -2)$$
$$w_1 (\alpha_{l-1}) = \alpha_l, \ \ w_1 (\alpha_l) = \alpha_{l-1}.$$
$$w_l(-\alpha_0) = \alpha_l, \ \ w_l(\alpha_l) = -\alpha_0, \ \ w_l(\alpha_i) = \alpha_{l-i} \ \ (1 \leq i \leq l-1). \ \ \ \ w_{l-1} = w_l w_1.$$

Consider the isogeny given by the subgroup $\langle \rho_l \rangle$ in $\Omega_{\mathrm{ad}}$.   Its associated cocharacter lattice, which we denote by $X_*(D_l^l)$, is given by $X_*(D_l^l) = \langle Q^{\vee}, \epsilon_l \rangle$.  We must compute $\mathcal{F}_{w_l}(1)$.  As $w_l$ exchanges $-\alpha_0$ and $\alpha_l$, $\mathcal{F}_{w_l}(1)$ is the set of positive roots that contain $\alpha_l$.  By \cite[Plate IV]{Bou02}, the sum of all of the (co)roots in $\mathcal{F}_{w_l}(1)$ is then equal to 
\[
(l-1)(\alpha_1 + 2 \alpha_2 + 3 \alpha_3 + ... + (l-2) \alpha_{l-2} + \frac{1}{2} (l-2) \alpha_{l-1} + \frac{1}{2} l \alpha_l).
\]
But notice that $\epsilon_l = \frac{1}{2}(\alpha_1 + 2 \alpha_2 + 3 \alpha_3 + ... + (l-2) \alpha_{l-2} + \frac{1}{2} (l-2) \alpha_{l-1} + \frac{1}{2} l \alpha_l).$ Therefore, $\mathcal{N}_{\circ}(w_l)^2 = (\epsilon_l)(-1)^{2(l-1)} = 1.$

We now consider the isogeny given by the subgroup $\langle \rho_1 \rangle$ of $\Omega_{\mathrm{ad}}$.  Its associated cocharacter lattice, which we denote by $X_*(D_l^1)$, is given by$X_*(D_l^1) = \langle Q^{\vee}, \epsilon_1 \rangle$.  Since $w_1$ exchanges $-\alpha_0$ and $\alpha_1$, $\mathcal{F}_{w_1}(1)$ is the set of all positive roots containing $\alpha_1$.  By \cite[Plate IV]{Bou02}, the sum of all of the (co)roots in $\mathcal{F}_{w_1}(1)$ is 
\[
(l-1)(2 \alpha_1 +  2 \alpha_2 +  2 \alpha_3 + ... +  2 \alpha_{l-2} + \alpha_{l-1} + \alpha_l).
\]
But notice that $\epsilon_1 = \alpha_1 + \alpha_2 + \alpha_3 + ... + \alpha_{l-2} + \frac{1}{2} \alpha_{l-1} + \frac{1}{2} \alpha_l$, so  $\mathcal{N}_{\circ}(w_1)^2 = (\epsilon_1)(-1)^{2(l-1)} = 1.$

We now consider the isogeny given by the subgroup $\langle \rho_{l-1} \rangle$ of $\Omega_{\mathrm{ad}}$, whose associated cocharacter lattice we denote by $X_*(D_l^{l-1})$.  To describe $\mathcal{F}_{w_{l-1}}(1)$, we must describe all positive roots that contain $\alpha_{l-1}$. By \cite[Plate IV]{Bou02}, the sum of all of the (co)roots in $\mathcal{F}_{w_{l-1}}(1)$ is $$(l-1)(\alpha_1 + 2 \alpha_2 + 3 \alpha_3 + ... + (l-2) \alpha_{l-2} + \frac{l}{2} \alpha_{l-1} + \frac{l-2}{2} \alpha_l).$$  But $X_*(D_l^{l-1}) = \langle Q^{\vee}, \epsilon_{l-1} \rangle$, and $\epsilon_{l-1} = \frac{1}{2} (\alpha_1 + 2 \alpha_2 + 3 \alpha_3 + ... + (l-2) \alpha_{l-2} + \frac{l}{2} \alpha_{l-1} + \frac{l-2}{2} \alpha_l).$
Thus, $\mathcal{N}_{\circ}(w_{l-1})^2 = (\epsilon_{l-1})(-1)^{2(l-1)} = 1$.

We now turn to $D_l$ with $l$ odd.  First we consider the adjoint case, denoting the associated cocharacter lattice by $X_*(D_l^{ad})$.  To show that $\mathcal{N}_{\circ}(w_l)^4 = 1$, we need to compute the sum 
\[
\gamma := \displaystyle\sum_{\alpha \in \mathcal{F}_{w_l}(1)} \alpha + \displaystyle\sum_{\beta \in \mathcal{F}_{w_l}(2)} \beta + \displaystyle\sum_{\delta \in \mathcal{F}_{w_l}(3)} \delta.
\] 
Noting that $$\mathcal{F}_{w_l}(1) = \{\alpha \in \Pi : \alpha \ \mathrm{contains} \ \alpha_{l-1} \ \mathrm{but \ doesn't \ contain} \ \alpha_1 \}$$
$$\mathcal{F}_{w_l}(2) =\{\alpha \in \Pi : \alpha \ \mathrm{contains} \ \alpha_1 \ \mathrm{but \ doesn't \ contain} \ \alpha_l \}$$
$$\mathcal{F}_{w_l}(3) =\{\alpha \in \Pi : \alpha \ \mathrm{contains} \ \alpha_l \},$$
one computes that 
\[
\gamma = 2(l-2)(\alpha_1 + 2 \alpha_2 + ... + (l-2) \alpha_{l-2}) + 2(\alpha_1 + \alpha_2 + ... + \alpha_{l-2}) 
\]
\[
 + ((l-2)(l-1) + 1) \alpha_{l-1} + \left(\frac{(l-3)(l-2)}{2} + \frac{(l-1)l}{2}\right) \alpha_l.
\]

Modulo $2 X_*(D_l^{ad})$, $\gamma$ is equivalent to $\alpha_{l-1} + \alpha_l.$  But $\alpha_{l-1} + \alpha_l \equiv 2\epsilon_1 \ (\mathrm{mod} \ 2X_*(D_l^{ad}))$, so  $\mathcal{N}_{\circ}(w_l)^4 = 1$ as needed.

We now consider the group $G$, of type $D_l$, with $l$ odd, that is neither simply connected nor adjoint.  We need to show that $\mathcal{N}_{\circ}(w_l^2)^2 = 1$.  First, we recall that $w_l^2$ fixes $\alpha_i$, for $i = 2, 3, ..., l-2$ and it exchanges $-\alpha_0$ and $\alpha_1$, and exchanges $\alpha_{l-1}$ and $\alpha_l$.  We must therefore count the positive roots that contain $\alpha_1$.  But this has already been computed in the $D_l$ cases with $l$ even, and our results there imply that $\mathcal{N}_{\circ}(w_l^2)^2 = 1$, noting that the cocharacter lattice in the current case is given by $\langle Q^{\vee}, \epsilon_1 \rangle$.

\item We now turn to the group $G$ of type $E_6$ and adjoint.  We follow here \cite[Plate V]{Bou02}, which has different conventions than \cite{IM65}.  The Weyl element $w$ in question acts by $\alpha_1 \mapsto \alpha_6 \mapsto -\alpha_0$.  Therefore, we need to compute the sum of all roots $\alpha$ that contain $\alpha_1$, together with all roots that contain $\alpha_6$ that also do not contain $\alpha_1$ .  One computes that this sum is
$$\displaystyle\sum_{\alpha \in \mathcal{F}_{w}(1)} \alpha^{\vee} + \displaystyle\sum_{\beta \in \mathcal{F}_w(2)} \beta^{\vee} = 16 \alpha_1 + 16 \alpha_2 + 24 \alpha_3 + 32 \alpha_4 + 24 \alpha_5 + 16 \alpha_6 \in 2 Q^{\vee}.$$  Therefore, $\mathcal{N}_{\circ}(w)^3 = 1$.

\item We now turn to $E_7$ adjoint.  We need to show that $\mathcal{N}_{\circ}(w)^2 = 1$, where $w$ is the Weyl element in question.  We follow here \cite[Plate VI]{Bou02}, which has different conventions than \cite{IM65}.  Using the fact that $w$ exchanges $\alpha_7$ and $-\alpha_0$, one counts that the sum of all of the positive roots that contain $\alpha_7$ is  
$$\displaystyle\sum_{\alpha \in \mathcal{F}_w(1)} \alpha^{\vee} = 18 \alpha_1 + 27 \alpha_2 + 36 \alpha_3 + 54 \alpha_4 + 45 \alpha_5 + 36 \alpha_6 + 27 \alpha_7.$$  But this sum is exactly equal to $18 \epsilon_7$, so $\mathcal{N}_{\circ}(w)^2 = ( \epsilon_7)(-1)^{18} = 1$.

\item We finally consider type $A_l$.  We re-adopt our conventions and notation from the proof of Proposition \ref{firstprop} in the case of type $A_l$.  That is, we let $a,b \in \mathbb{N}$ such that $l + 1 = ab$.  Let $\omega = \rho_1^a$, so that $\omega^b = 1$, and for ease of notation, let $n = l+1$.  Then $\omega = \epsilon_a w_a$, where $w_a$ is the a-th power of the $n$-cycle $(1 \ 2 \ \cdots \ n)$.  

A computation then shows that $$\displaystyle\sum_{m = 1}^{b-1} \displaystyle\sum_{\alpha \in \mathcal{F}_{w_a}(m)} \alpha^{\vee} = (n-a)[e_1 + e_2 + ... + e_a] + (n-3a)[e_{a+1} + e_{a+2} + ... + e_{2a}] $$ $$+ (n-5a)[e_{2a+1} + ... + e_{3a}]  + ... +(a-n)[e_{n-a+1} + e_{n-a+2} + ... + e_{n-1} + e_n].$$ We denote this sum by $\gamma$.  We recall that the cocharacter lattice of this isogeny is given by $X_*(A_l^a) = \langle Q^{\vee}, \epsilon_a \rangle$, where 
$$\epsilon_a = \frac{1}{n}[ (n-a)(e_1 + e_2 + ... + e_a) - a(e_{a+1} + ... + e_n)].$$
Suppose first that $n$ is odd, so that $a$ is also odd.  Therefore, $n - a, n - 3a, n - 5a, ..., a - n$ are all even, so one can see that $\gamma \in 2 Q^{\vee}$, which implies that $\mathcal{N}_{\circ}(w_a)^b = 1$.
Suppose now that $n$ is even.  Then 
\[
\gamma - n \epsilon_a = (n-2a)[e_{a+1} + e_{a+2} + ... + e_{2a}] + (n-4a)[e_{2a+1} + ... + e_{3a}]
\]
\[
+ ... +(2a-n)[e_{n-a+1} + e_{n-a+2} + ... + e_{n-1} + e_n].
\] 
One can see that $\gamma - n \epsilon_a =: \eta \in 2 Q^{\vee}$.  Therefore, $\gamma = \eta + n \epsilon_a$ is twice a cocharacter, so $\mathcal{N}_{\circ}(w_a)^b = 1$.
\end{itemize}
\end{proof}

\begin{remark}
The previous argument in the case of type $A_l$ depends on the group not being simply connected.  Otherwise, it may not be that $\mathcal{N}_{\circ}(w_a)^b = 1$.  Indeed, if $n$ is even, we relied on the fundamental coweight $\epsilon_a$ being contained in the cocharacter lattice in order to conclude that $\mathcal{N}_{\circ}(w_a)^b = 1$.  If $n$ is odd, however, it was automatic that $\mathcal{N}_{\circ}(w_a)^b = 1$.  Indeed, this does not conflict with a basic known example; if $G = SL(n)$ and $w$ is the long Weyl element, then 
\begin{equation*}
\mathcal{N}_{\circ}(w)^n = \left\{
\begin{array}{lll}
-1 & \text{if} &  n \ \text{is even}\\
1 &  \text{if} & n \ \text{is odd}
\end{array} \right.
\end{equation*}
\end{remark}

\begin{theorem}\label{mainresult}
For $G$ a split, almost simple, $p$-adic group, there exists an embedding $\Omega \hookrightarrow N_G(T)$ that is also a section of the canonical map $N_G(T) \rightarrow \Omega$.
\end{theorem}

\begin{proof}
Suppose that $\Omega$ is cyclic of order $n$.  We have shown that if $\omega = \epsilon_i w_{\Delta_i} w_{\Delta} \in \Omega$ is a generator, then $\iota(\omega)^n = 1$.  But in fact $\iota(\omega)$ has order $n$.  To see this, note that if $m < n$, then $\iota(\omega)^m = (\epsilon_i + w_i(\epsilon_i) + ... + w_i^{m-1}(\epsilon_i))(\varpi^{-1}) \cdot \mathcal{N}_{\circ}(w_i)^m$.  But $\mathcal{N}_{\circ}(w_i)^m$ has a nontrivial projection to $W_{\circ}$ since $w_i$ has order $n$.  Therefore, $\iota(\omega)^m$ has a nontrivial projection to $W_{\circ}$ as well, so in particular must be nontrivial.  Since $\iota(\omega)$ has order $n$, we may define a homomorphism $\Omega \rightarrow N_G(T)$ by sending $\omega^j$ to $\iota(\omega)^j$, and one may check that this map is in fact a section of the map $N_G(T) \rightarrow \Omega$.

It remains to consider the case where $G$ is adjoint of type $D_l$ with $l$ even, since its fundamental group is not cyclic.  We denote the associated cocharacter lattice by $X_*(D_l^{ad})$.  We show that $\iota$ is a homomorphism in this case.  Recall that in Proposition \ref{secondprop}, we showed that $\iota(\omega)^2 = 1$ for each $\omega \in \Omega$.  We need to show that $\iota(\rho_1 \rho_l) = \iota(\rho_1) \iota(\rho_l), \iota(\rho_1 \rho_{l-1}) = \iota(\rho_1) \iota(\rho_{l-1}),$ and $\iota(\rho_{l-1} \rho_l) = \iota(\rho_{l-1}) \iota(\rho_l)$. We will carry out the case $\iota(\rho_1 \rho_l) = \iota(\rho_1) \iota(\rho_l)$, noting that the other cases are similar. First note that $\iota(\rho_1 \rho_l) = \iota(\rho_{l-1}) = \epsilon_{l-1} \mathcal{N}_{\circ}(w_{l-1})$ and $\iota(\rho_1) \iota(\rho_l) = \epsilon_1 \mathcal{N}_{\circ}(w_1) \epsilon_l \mathcal{N}_{\circ}(w_l) =  \epsilon_1 \mathcal{N}_{\circ}(w_1) \epsilon_l \mathcal{N}_{\circ}(w_1)^{-1} \mathcal{N}_{\circ}(w_1) \mathcal{N}_{\circ}(w_l)$.  One can compute that $\epsilon_1 \mathcal{N}_{\circ}(w_1) \epsilon_l \mathcal{N}_{\circ}(w_1)^{-1} = \epsilon_{l-1}$, so it suffices to show that $\mathcal{N}_{\circ}(w_1) \mathcal{N}_{\circ}(w_l) = \mathcal{N}_{\circ}(w_{l-1})$.  By Corollary \ref{ahomomorphicity}, we need to show that $\displaystyle\prod_{\alpha \in \mathcal{F}(w_1,w_l) } \alpha^{\vee}(-1) = 1.$  One computes that
$$\mathcal{F}(w_1, w_l) = \{\alpha \in \Pi : w_l(\alpha) \in -\Pi, w_1 w_l(\alpha) \in \Pi \} = \{\alpha \in \Pi : \alpha \ \mathrm{contains \ } \alpha_l \ \mathrm{and \ } \alpha \ \mathrm{does \ not \ contain} \ \alpha_{l-1} \}.$$

This last set, by \cite[Plate IV]{Bou02}, is the set $\{e_i + e_l : 1 \leq i < l \}$.  Adding these roots together gives $\gamma := \alpha_1 + 2 \alpha_2 + 3 \alpha_3 + ... + (l-2) \alpha_{l-2} + (l-1) \alpha_l$.  But $X_*(D_l^{ad})$ contains $\epsilon_{l-1}, \epsilon_l$, and we see that $\gamma = (2-l) \epsilon_{l-1} + l \epsilon_l$, which lives in $2 X_*(D_l^{ad})$ since $l$ is even.  The result follows.
\end{proof}

\begin{remark}\label{nonhomomorphic} $ $
\begin{enumerate} 
\item It is not difficult to show that $\iota$ is a homomorphism in the case that $G$ is adjoint of type $E_6$.  Since we cannot claim this for all types, we do not include the computation.
\item In the case that $G$ is adjoint of type $D_l$ where $l$ is odd, one can show that $\iota$ is not a homomorphism.  In fact, one can show that $\iota(\rho_l)^2 = \iota(\rho_l^2)$, but it turns out that $\iota(\rho_l)^3 \neq \iota(\rho_l^3)$.  This boils down to computing that the sum of all (co)roots in $\mathcal{F}_{w_l}(2)$ equals $le_1 - (e_1 + e_2 + ... + e_l)$, which when evaluated at $-1$ is nontrivial.
\item In the case that $G$ is type $A_l$, it turns out that $\iota$ is sometimes a homomorphism and sometimes not.  For example, if $a = 1$ (in the notation of Proposition \ref{secondprop}), then the group in consideration if $PGL_n$ (recall that in our notation, $n = l+1 = ab$), and one can show that $\iota(\rho_1)^2 \neq \iota(\rho_1^2)$.  On the other hand, if both $n$ and $a$ are even, then $\iota$ is a homomorphism.
\end{enumerate}
\end{remark}

\section{Beyond split almost-simple groups}
One may ask about generalizing Theorem \ref{mainresult} to more general connected reductive groups.   The biggest obstacle to generalizing the result, using the methods in this paper, revolves around the fact that if $W_{\circ}(\Omega)$ denotes the projection of $\Omega$ onto the finite Weyl group, then $\mathcal{N}_{\circ}|_{W_{\circ}(\Omega)} : W_{\circ}(\Omega) \rightarrow N_G(T)$ is not necessarily a homomorphism.  This problem occurred in some $A_l$ types, as well as adjoint $D_l$ with $l$ odd.  But in these cases, we were able to skirt this issue by adjusting $\iota$ as in Theorem \ref{mainresult}, using the fact that $\Omega$ is cyclic.

On the other hand, we are able to extend our result to certain additional split connected reductive groups.  Note first that since $G_{\mathrm{ad}}$ is a product of split, almost-simple groups, Theorem \ref{mainresult} gives a section $s_{G_{\mathrm{ad}}}$ of $\kappa_{G_{\mathrm{ad}}} : G_{\mathrm{ad}}(F) \rightarrow \Omega_{G_{\mathrm{ad}}}$.

\begin{definition}
Call a homomorphic section $s_G$ of $\kappa_G$ \emph{good} if it is compatible with the one constructed for $G_{\mathrm{ad}}$.  In other words, the following diagram commutes:
\[ \begin{tikzcd}
 \arrow[swap]{d}{\kappa_G} G(F)  \arrow{r} & G_{\mathrm{ad}}(F) \arrow[swap]{d}{\kappa_{G_{\mathrm{ad}}}}   \\%
 \arrow[swap, shift right=2]{u}{s_G}  \Omega_G \arrow{r}  &  \arrow[swap, shift right=2]{u}{s_{G_{\mathrm{ad}}}} \Omega_{\mathrm{ad}} 
\end{tikzcd}
\]
\end{definition}

\begin{remark}
Recall that when $G_{\mathrm{der}} = G_{\mathrm{sc}}$, there is an easy way to produce a homomorphic section with values in $T(F)$.  However, this will not generally make the diagram commute, so it is not good.
\end{remark}

\begin{proposition}\label{connectedcenter}
Let $G$ be a split connected reductive group over $F$.  Let $C$ be an alcove in the apartment corresponding to a split maximal torus $T$, with associated extended affine Weyl group $W = X_*(T) \rtimes W_{\circ}$.  Then:
\begin{enumerate}
\item If $Z = Z(G)$ is connected, then the induced map $G(F) / Z(\mathcal{O}_F) \rightarrow \Omega_G$ has a good homomorphic section (the analogue of the diagram above commutes).
\item If $Z$ is connected and $\Omega_G \cong \mathbb{Z}$ (e.g. $G = GSp(2n)$), then $\kappa_G$ has a good homomorphic section.
\item If $Z$ is connected, $\Omega_G \cong \mathbb{Z}^n$, with $n > 1$, and $(|\Omega_{G_{\mathrm{ad}}}|, q(q-1)) = 1$, where $q$ is the cardinality of the residue field, then $\kappa_G$ has a good homomorphic section.
\end{enumerate}
\end{proposition}

\begin{proof}
We start with (1).  It follows from Theorem \ref{mainresult} that $\kappa_{G_{\mathrm{ad}}}$ has a homomorphic section $s_{G_{\mathrm{ad}}}$, since $G_{\mathrm{ad}}$ is known to be a product of almost-simple groups.  Moreover, if $\kappa_Z$ denotes the Kottwitz homomorphism for $Z(F)$, then $\kappa_Z$ also has a homomorphic section, which we denote $s_Z$.  As $H^1(F,Z) = 1$, we have a commutative diagram of exact sequences

\[ \begin{tikzcd}
1 \arrow{r} & Z(F) \arrow{r} \arrow[swap]{d}{\kappa_Z} &  \arrow[swap]{d}{\kappa_G} G(F) \arrow{r}{\mathrm{pr}} & G_{\mathrm{ad}}(F) \arrow[swap]{d}{\kappa_{G_{\mathrm{ad}}}} \arrow{r} & 1  \\%
1 \arrow{r} & \Omega_Z \arrow{r}& \Omega_G \arrow{r}{\mathrm{pr}}  & \Omega_{G_{\mathrm{ad}}} \arrow{r} & 1
\end{tikzcd}
\]

We naturally have $Z(F) / Z(\mathcal{O}_F) \cong X_*(Z)$, therefore obtaining another diagram

\[ \begin{tikzcd}
1 \arrow{r} & Z(F)/ Z(\mathcal{O}_F) \arrow{r} \arrow[swap]{d}{\overline{\kappa_Z}} &  \arrow[swap]{d}{\overline{\kappa_G}} G(F) / Z(\mathcal{O}_F) \arrow{r}{\overline{\mathrm{pr}}} & G_{\mathrm{ad}}(F) \arrow[swap]{d}{\kappa_{G_{\mathrm{ad}}}} \arrow{r} & 1  \\%
1 \arrow{r} & \Omega_Z \arrow{r}& \Omega_G \arrow{r}{\mathrm{pr}}  & \Omega_{G_{\mathrm{ad}}} \arrow{r} & 1
\end{tikzcd}
\]

where $\overline{\kappa_Z}, \overline{\kappa_G}$ are the induced maps.  Let $\overline{\kappa_{G_{\mathrm{ad}}}}$ denote the map induced from $\kappa_{G_{\mathrm{ad}}}$ on $s_{\mathrm{ad}}(\Omega_{\mathrm{ad}})$.  Then we have a commutative diagram of groups:

\[ \begin{tikzcd}
1 \arrow{r} & Z(F) / Z(\mathcal{O}_F) \arrow{r} \arrow[swap]{d}{\overline{\kappa_Z}} &  \arrow[swap]{d}{\overline{\kappa_G}} \overline{\mathrm{pr}}^{-1}(s_{\mathrm{ad}}(\Omega_{\mathrm{ad}})) \arrow{r}{\mathrm{pr}} & s_{\mathrm{ad}}(\Omega_{\mathrm{ad}})  \arrow[swap]{d}{\overline{\kappa_{G_{\mathrm{ad}}}}} \arrow{r} & 1  \\%
1 \arrow{r} & \Omega_Z \arrow{r}& \Omega_G \arrow{r}{\mathrm{pr}}  & \Omega_{G_{\mathrm{ad}}} \arrow{r} & 1
\end{tikzcd}
\]

We have that $\overline{\kappa_Z}, \overline{\kappa_{G_{\mathrm{ad}}}}$ are isomorphisms, so by the five lemma, $\overline{\kappa_G}$ is an isomorphism, and thus the map $\overline{\kappa_G} : G(F) / Z(\mathcal{O}_F) \rightarrow \Omega_G$ has a homomorphic section. 

We now prove (2).  Make an initial choice of a homomorphic section $s_Z^0$ of $\kappa_Z$.  Given $\sigma \in \Omega_G$, let $s^0(\sigma)$ be any lift in $G(F)$ of $s_{G_{\mathrm{ad}}}(\mathrm{pr}(\sigma)) \in N_{G_{\mathrm{ad}}}(T_{\mathrm{ad}})(F)$; it automatically lies in $N_G(T)(F)$.  It might happen that $s^0$ is not a section of $\kappa_G$.  However, for all $\sigma \in \Omega_G$, we have $\mathrm{pr}(\kappa_G(s^0(\sigma))) = \kappa_{G_{\mathrm{ad}}}(\mathrm{pr}(s^0(\sigma))) = \kappa_{G_{\mathrm{ad}}}(s_{G_{\mathrm{ad}}}(\mathrm{pr}(\sigma))) = \mathrm{pr}(\sigma)$.  Thus, the difference between $\sigma$ and $\kappa_G(s^0(\sigma))$ belongs to $\Omega_Z$.  Since $\kappa_Z$ is surjective, we may alter each $s^0(\sigma)$ by an element $z_{\sigma}^0 \in Z(F)$ in such a way that $\sigma \mapsto s^0(\sigma) z_{\sigma}^0$ is a section of $\kappa_G$.

So we may assume $s^0$ is a set-theoretic section of $\kappa_G$, taking values in $N_G(T)(F)$.  Because $s_{G_{\mathrm{ad}}}$ is homomorphic, the map $$(\sigma_1, \sigma_2) \mapsto s^0(\sigma_1) s^0(\sigma_2) s^0(\sigma_1 \sigma_2)^{-1}$$ is a $2$-cocycle of $\Omega_G$ with values in $Z(F)$, with $\Omega_G$ acting trivially on $Z(F)$.  Therefore, we get an element of $H^2(\Omega_G, Z(F))$.  This group parameterizes isomorphism classes of extensions of $\Omega_G$ by $Z(F)$ where the induced action of $\Omega_G$ on the normal subgroup $Z(F)$ is trivial (i.e. $Z(F)$ is central in the extension group).  We claim that the extension corresponding to the $2$-cocycle is the direct product $Z(F) \times \Omega_G$.  This follows because $\Omega_G = \mathbb{Z}$ and $H^2(\mathbb{Z}, A) = 1$ for any abelian group $A$ with trivial $\mathbb{Z}$-action.

The fact that the extension is trivial means that the $2$-cocycle defining it is a $2$-coboundary.  This means that we may alter our initial choice of set-theoretic section $s^0$ to give a homomorphism $s : \Omega_G \rightarrow G(F)$, taking values again in $N_G(T)(F)$.

The problem now is that $s$ might not be a section of $\kappa_G$, which we take care of as before.  By construction, $\sigma^{-1} \kappa_G(s(\sigma)) \in \Omega_Z$ for every $\sigma \in \Omega_G$.  So we may define $z_{\sigma} := s^0_Z(\sigma  (\kappa_G(s(\sigma)))^{-1}) \in Z(F)$, for $\sigma \in \Omega_G$.  Note that $\sigma \mapsto z_{\sigma}$ is a homomorphism $\Omega_G \rightarrow Z(F)$.  Now define $$s_G(\sigma) := z_{\sigma}  s(\sigma).$$  Then $s_G$ is the desired homomorphic section of $\kappa_G$ in case (2).

In case (3), the same argument works, as long as we can prove that the $2$-cocycle defined by $s^0$ is still a $2$-coboundary.   But when $n >1$ it is no longer true that $H^2(\mathbb{Z},A)$ always vanishes for abelian groups $A$ with trivial $\mathbb{Z}^n$-action.  Nevertheless, we will show that the extension corresponding to the given $2$-cocycle is still trivial.  Write $\dot{e}_i = s^0(e_i)$, where $e_i$ corresponds to a standard basis vector in $\Omega_G \cong \mathbb{Z}^n$.  Then the extension is the exact sequence 
$$1 \rightarrow Z(F) \rightarrow Z(F) \langle \dot{e}_1, \cdots, \dot{e}_n \rangle \xrightarrow{\kappa_G} \Omega_G \rightarrow 1.$$
Write $N := |\Omega_{G_{\mathrm{ad}}}|$.  As $\mathrm{pr}(\dot{e}_j) \in \mathrm{im}(s_{G_{\mathrm{ad}}}) \cong \Omega_{G_{\mathrm{ad}}}$, we have $\mathrm{pr}(\dot{e}_j)^N = 1$ and hence $\dot{e}_j^N \in Z(F)$.  Moreover, $\dot{e}_i \dot{e}_j \dot{e}_i^{-1} \dot{e}_j^{-1} \in Z(F)$.  We may write 
$$a \dot{e}_j = \dot{e}_i \dot{e}_j \dot{e}_i^{-1},$$ for some $a \in Z(F)$.  Raising to the $N$-th power, we get $$a^N \dot{e}_j^N = \dot{e}_j^N,$$ and hence $a^N = 1$.  Therefore, $a \in Z(\mathcal{O}_F)$.  Moreover, since $N$ is coprime to the pro-order of the profinite group $Z(\mathcal{O}_F)$, we conclude that $a = 1$, and therefore the elements $\dot{e}_i$ pairwise commute.  Therefore, the extension is an abelian group.  But then the extension is trivial, since $\Omega_G \cong \mathbb{Z}^n$.  

This concludes the proof of the proposition.  But we make one additional comment.  By construction, the map $s_G|_{\Omega_Z}$ has image in $Z(F)$ and so gives a homomorphic section $s_Z$ of $\kappa_Z$.  This section might be different from the initial choice $s^0_Z$.  But now we have a commutative diagram

\[ \begin{tikzcd}
1 \arrow{r} & Z(F) \arrow{r} \arrow[swap]{d}{\kappa_Z} &  \arrow[swap]{d}{\kappa_G} G(F)  \arrow{r} & G_{\mathrm{ad}}(F) \arrow[swap]{d}{\kappa_{G_{\mathrm{ad}}}} \arrow{r} & 1  \\%
1 \arrow{r} &  \arrow[swap, shift right=2]{u}{s_Z} \Omega_Z \arrow{r}& \arrow[swap, shift right=2]{u}{s_G}  \Omega_G \arrow{r}  &  \arrow[swap, shift right=2]{u}{s_{G_{\mathrm{ad}}}} \Omega_{\mathrm{ad}}  \arrow{r} & 1
\end{tikzcd}
\]

\end{proof}

\end{document}